\documentclass{amsart}
\usepackage{ adjustbox }
\usepackage{ amsthm }
\usepackage{ amssymb }
\usepackage{ amsmath }
\usepackage{ array }
\usepackage{ booktabs }
\usepackage{ enumitem }
\usepackage{ latexsym }
\usepackage{ url }
\usepackage{letltxmacro} 
\LetLtxMacro{\oldsqrt}{\sqrt}  
\renewcommand{\sqrt}[2][\mkern8mu]{\mkern-4mu\mathop{\oldsqrt[#1]{#2}}}
\usepackage[all]{ xy }
\usepackage{ tikz-cd }
\usetikzlibrary{ trees}
\usetikzlibrary{ arrows, calc, positioning }
\usetikzlibrary{ decorations.pathmorphing }
\usetikzlibrary{ decorations.markings }

\theoremstyle{definition}

\newtheorem{definition}{Definition}[section]

\newtheorem{example}[definition]{Example}

\newtheorem{remark}[definition]{Remark}

\theoremstyle{plain}

\newtheorem{corollary}[definition]{Corollary}
\newtheorem{lemma}[definition]{Lemma}

\newtheorem{theorem}[definition]{Theorem}

\makeatletter
\def\l@section{\@tocline{1}{0pt}{1pc}{}{}}
\def\l@subsection{\@tocline{2}{0pt}{1pc}{4.6em}{}}
\def\l@subsubsection{\@tocline{3}{0pt}{1pc}{7.6em}{}}
\renewcommand{\tocsection}[3]{%
  \indentlabel{\@ifnotempty{#2}{\makebox[1.25em][l]{\ignorespaces#1#2.}}}#3}
\renewcommand{\tocsubsection}[3]{%
  \indentlabel{\@ifnotempty{#2}{\hspace*{1.25em}\makebox[2.00em][l]{\ignorespaces#1#2.}}}#3}
\renewcommand{\tocsubsubsection}[3]{%
  \indentlabel{\@ifnotempty{#2}{\hspace*{3.25em}\makebox[2.75em][l]{\ignorespaces#1#2.}}}#3}
\makeatother

\allowdisplaybreaks

\begin{document}

\title{Classification of  finite W-groups}

\author{Fatemeh Bagherzadeh}

\address{Department of Mathematics and Statistics, University of Saskatchewan, Canada}

\email{bagherzadeh@math.usask.ca}



\thanks{
The results of this paper form part of the author's Ph.D. thesis, completed in 2014
under the supervision of Prof. J\'a{n}  Min\'a\v{c} of the Department of Mathematics at Western University, Canada.
This paper was written while the author was a Postdoctoral Fellow in the Department of 
Mathematics and Statistics at the University of Saskatchewan, Canada, partially supported by 
the NSERC Discovery Grant of Prof. Murray Bremner.}

\begin{abstract}
We determine the structure of  the W-group $\mathcal{G}_F$, the small Galois quotient of the absolute Galois group $G_F$ of  the Pythagorean formally real field  $F$ when the space of orderings $X_F$ has finite order. 
Based on  Marshall's work (1979), we reduce the 
structure of $\mathcal{G}_F$ to that of $\mathcal{G}_{\overline{F}}$, the W-group of the 
residue field $\overline{F}$ when $X_F$ is a connected space.
In the disconnected case, the structure of $\mathcal{G}_F$ is the free product of  the
W-groups $\mathcal{G}_{F_i}$ corresponding to the connected components 
$X_i$ of $X_F$. 
We also give a completely Galois theoretic proof for  Marshall's Basic Lemma.
\end{abstract}

\maketitle


%
\section{Introduction} 
The characterization of Witt rings in the category of all rings is a very difficult problem. 
Currently  we  only have information  about finitely generated Witt rings. 
For each field $F$ of characteristic not two, there exists  a certain Galois group which carries the same information as the Witt ring of $F$; see \cite{W1979}, \cite[Theorem 3.8]{MS1996}.
This group was denoted by $\mathcal{G}_F$ and called  the W-group of $F$; see \cite{MS1990}.
Our approach for determining the structure of W-groups,  which are small  quotients of the absolute Galois groups $G_F$, is based on  Marshall's theory of spaces of orderings \cite{M1979}, but from the Galois theoretic point of  view.\\

We focus on W-groups of some Pythagorean fields for which  $-1{\notin}\sum F^2$ and, especially, on   their spaces of orderings.
A \textit{Pythagorean} field is a field in which every sum of two squares is a square. 
The two obvious  examples of Pythagorean fields are $\mathbb{R}$ and $\mathbb{C}$.
Artin  and  Schreier called any field with the property $-1{\notin}\sum F^2$, \textit{formally real} and  they characterized them as fields admitting orderings; see \cite[Theorem 1.5]{L1983}.
The theory of formally real fields originally comes from  a basic algebraic property of the field of real numbers which is that the only relations  of the form $\sum \alpha_i^2=0$ are the trivial ones $ 0^2+0^2+,\dots ,+0^2=0$.\\

 Throughout this paper  $F$ is a field with $\hbox{char}F=0$.
The group of nonzero elements of the  field $F$  is denoted by $\dot{F}$ and $\dot{F}/\dot{F^2}$ is the  square class group, which is an  $\mathbb{F}_2$-vector space.

 \begin{definition}\cite[\S.1]{L1983}
 An \textit{ordering} of the field $F$ is a subset $P \subsetneqq F$ such that $P+P \subseteq P$,  $P \cdot P \subseteq P$ and $P \cup (-P)=F$.
 A \textit{preordering} of  the field $F$ is a proper subset $T\subsetneqq F$ such that $\dot{F}^2 \subseteq T$, $T+T \subseteq T$ and $ T\cdot T \subseteq T$. 
 \end{definition}
 
 Let $\dot{P}=P{-}\{0\}$, then $\dot{P}$ is a subgroup of $\dot{F}$, $\sum \dot{F}^2 \subset P$ and $-1{\notin} P$. 
If $F$ is a  Pythagorean formally real field, $F^2$  is a preordering of $F$. 
~We keep the notation of \cite{MS1990, MS1996}, so  $F(2)$ is the  \textit{quadratic closure} of $F$ and $G_F:=Gal \big(F(2)/F \big)$. 
Now for a field $F$ let $F^{(2)}=F(\sqrt{a}: a \in \dot{F})$,  the compositum of all quadratic extensions of $F$. 
So $F^{(2)}$ is the smallest field  that contains all square roots 
   $\sqrt{a}$ for all $a\in \dot{F}$. Let $F^{(3)}=F^{(2)}(\sqrt{y}: y\in F^{(2)})$ such that $ F^{(2)}(\sqrt{y})/F$ is Galois. 
So $F^{(3)}$ is  the compositum of all quadratic extensions $K$ of $F^{(2)}$ that $ K/F$ is Galois.

\begin{definition}\cite[\S2]{MS1996}
   Let $F$ be a Pythagorean formally real field, the \textit{W-group} of  $F$ is defined by $\mathcal{G}_F=Gal \big(F^{(3)}/F \big)$.
\end{definition}
\subsection{Overview of Problem, Method and Results}
Our goal is  to determine the structure of the W-group of the field $F$ when $F$ is a Pythagorean formally real  field with finitely many orderings. 
One of the main theorem we use relates  $X_F$, the  set of orderings  of the field $F$, to the set of non simple involutions  $ \{\sigma \in \mathcal{G}_F\, | \, \sigma^2=1,\, \sigma\notin \Phi(F)\}$; see \cite[corollary 2.10]{MS1990}. 
 By using Galois theory we  show $X_F$ is a  space of orderings (Theorem \ref{xFspace of oredring})  and   we also present  the Galois theoretic proof for the Basic Lemma; see \cite[Lemma 1.3]{M1979} which make this mysterious lemma  more clear. 
Thereby  all techniques about the general  space of orderings can be used for  classifying the W-group $\mathcal{G}_F$. 
By using Marshall's  method for  classifying  the general space of orderings (see \cite{M1979}) we  show if  the space of orderings of the field $F$ is connected, the translation group of $\mathcal{G}_F$ is nontrivial, so the structure of $\mathcal{G}_F$ will be reduced  to the structure of $\mathcal{G}_{\overline{F}}$,  (Theorem \ref{Connected case}. $\overline{F}$ is the residue  field of $F$). 
If $X_F$ is disconnected, the structure of $\mathcal{G}_F$ will be reduced to the free product of $\mathcal{G}_{F_i}$ such that $\mathcal{G}_{F_i}$ are W-groups corresponding to the connected components $X_i$ of $X_F$ (Theorem \ref{disconnectedcase}).
\section{Preliminaries}
\subsection{Space  of Orderings} 
\begin{definition}\label{orderingspace}\cite{M1979}
A pair $(X,G)$ such that $G$ is an elementary 2-group, ${-}1$ is a distinguished  element in  $G$ and  $X$ is a subset of the character group $\chi(G)=Hom(G,\{1,{-}1\})$ is called a \textit{space of orderings} if it satisfies the following properties:
\begin{enumerate}
\item $ X$ is a closed subset of $ \chi(G)$.
\item For any ordering $\sigma \in X,\; \sigma(-1)={-1}$.
\item $X^\bot=\{a \in G\,|\,\sigma a=1\, \forall  \,\sigma \in X\}=1$.
\item If  $f$ and $g$ are two forms over $G$ and $ x\in D_{f\oplus g}$, then there exist $y\in D_f$ and $z \in D_g$ such that $x \in D_{\langle y,z\rangle}$.
\end{enumerate}
\end{definition}

\begin{example}
Let $F$ be a formally real field, $X_F$ the set of all orderings of $F$ and $G=\dot{F}/\sum \dot{F^2}$. 
We  show $(X_F,G_F)$ is a  space of orderings, (Theorem \ref{xFspace of oredring}).
\end{example}

 Let  $\sigma_i \in X$ for $i=1,\dots, m$, $\epsilon_i \in \{ 0,1 \}$ and 
\[
Y=\{ \sigma  \in X \, |  \,\sigma  = {\sigma_1}^{\epsilon_1}, \dots, {\sigma_m}^{\epsilon_m}\},\qquad
\Delta =\{a \in G\, | \, \sigma_i(a)=1 \}.
\]
Then  $(Y,G/ \Delta)$ is a \textit{subspace}  of $(X,G)$ generated by $\sigma_1,\dots,\sigma_m $ if $Y$ and $\Delta$   satisfy the duality condition; $\Delta = Y^\bot$ and  $Y= {\Delta^\bot} \cap X$; see \cite{M1979}. 
Note that $\sigma$ is  a product of an odd number of orderings. 
The subspace  $(Y,G/ \Delta)$  of $(X,G)$ is also a space of orderings; see \cite[Lemma 2.2]{M1979}.

\begin{remark}
 Two spaces of orderings $(X, G)$ and $(X', G')$ are  called equivalent, denoted by $(X, G) \backsim (X, G')$, if there exists a group isomorphism $\varphi: G \cong G'$ such that the dual isomorphism $\varphi^* : \chi(G')\rightarrow \chi(G)$ maps $X'$ onto $X$; see  \cite{M1979}. 
\end{remark}

\begin{definition}\cite{M1979}\label{connectedorderings}
 Two orderings $\sigma, \sigma'$ of  $X$ are called \textit{simply connected} in $X$, denoted by $\sigma\backsim_s\sigma'$, if there exist orderings $\tau,\tau'$ of $X$ such that  $\sigma\sigma'=\tau\tau'$ and  $\{ \sigma,\sigma' \} \neq \{ \tau,\tau' \}$.
~If either $\sigma=\sigma'$ or there exists a sequence of orderings 
$\sigma=\sigma_0,\,\sigma_1,\dots,\sigma_k =\sigma' \in X$ such that  $\sigma_{i-1}\backsim_s\sigma_i$ for all $ i=1,\dots,k$, then $\sigma$ and $\sigma'$ are  \textit{connected} and  denoted  by $\sigma\backsim\sigma'$. 
\end{definition}

Let $X_i$,   $i=1,\dots,k$ be the equivalence classes of connected components  and $X=X_1\cup, \dots ,\cup X_k$ the decomposition of $X$  determined by the relation $\sim$.
Any of $X_i$ are called the connected components of $X$.
The  ordering space $X$ is connected if it has only one connected component and  the rank (or dimension) of $X$ is $n$ if there exists a basis  $\{\sigma_1,\dots,\sigma_n \}$ of $\chi(G)$ generating the  elements of $X$,  any element of $X$ should be a product of an odd number of $\sigma_i$.

\begin{theorem}\label{DecompositionTheorem} \cite{M1979}
Suppose $X_1,\dots,X_k$ are  the connected components of $X$, then any of  $X_i$ is  a subspace of $X$ and $\hbox{rank}\,X=\sum^k_1 \hbox{rank}\,X_i.$
\end{theorem}

\begin{remark}\label{familyofxalpha}
For $\alpha \in \chi(G)$ let $X_\alpha = \{\sigma\in X\, | \, \sigma\alpha \in X\}=\alpha X\cap X$. Indeed $X_\alpha$ is a subspace of $X$  and  the maximal subset of $X$  satisfying $\alpha X_\alpha =X_\alpha$; see  \cite[Lemma 4.3]{M1979}. 
Consider the family  $\mathcal{M}=\{X_\alpha\,| \, \alpha \in \chi(G)\}$. 
We recall  briefly the   following propositions about the elements of $\mathcal{M}$; see \cite[ Lemma  4.4]{M1979}: 
\begin{enumerate}

\item If $\sigma_1,\sigma_1\alpha,\sigma_1\beta,\sigma_1\alpha\beta$ are different elements in $ X_F$, then either $X_\alpha\subseteq X_\beta$ or $X_\beta\subseteq X_\alpha$.

\item If $\alpha\neq 1$, $\beta\neq 1$, $ X_\alpha\cap X_\beta\neq\emptyset$, rank $ X_\alpha\ge3$ and  rank $X_\beta \ge3$, then there exists $\gamma\in \chi(G)$, $\gamma\neq1$ such that $X_\alpha$, $X_\beta\subseteq X_\gamma$.
\item
If rank $X_\alpha \geqslant 3$ and rank $ X_\beta\geqslant 3$, then either $X_\alpha\cap X_\beta=\emptyset$ or  $|X_\alpha\cap X_\beta|\geqslant 2$. 
\end{enumerate}
\end{remark}

\begin{theorem}\cite{M1979}
 If $X$ is a connected space with rank $X\neq1$, then there exists $\alpha \in \chi(G)$, $\alpha\neq1$, such that $\alpha X=X$.
\end{theorem}

\begin{definition}\label{translationgroup}
For  spaces  of orderings $X$, we  define the \textit{translation group}  $T$ as the set of all $\alpha \in \chi(G)$  such that $\alpha X=X$.
\end{definition}

\begin{theorem}\cite[Theorem 4.8]{M1979}\label{spaceX'}
If $X$ is a connected space of orderings, $T$ the translation group of $X$, $G'=T^\bot$ and  $X'$ the set of all restrictions  $\sigma|_ {G'}$ such that $\sigma\in X$, then $(X',G')$ is a space of orderings.
\end{theorem}

\begin{theorem} \cite[Theorem 4.10]{M1979}\label{existF}
Let $(X,G)$ be a finite space of orderings. 
Then there exists a Pythagorean formally real field $F$ such that $(X,G){\backsim} (X_F,G_F)$.
\end{theorem}

 Marshall  proved that if $X_F$ is a connected space of orderings, then the  translation group $T$ is a nonempty subgroup  of $\chi(G)$; see \cite{M1979}. 
By using  this,  he   reduced  the study of the structure of connected  spaces of orderings to  ordering spaces with  smaller rank.  
In this work we use  Galois  theoretic techniques  to show that  when the set of orderings of  the field $F$ is connected,  the W-group of $F$ has a nontrivial translation group modulo the  Frattini subgroup.
\subsection{W-group of Pythagorean Formally Real Fields}
 A \textit{quadratic form}  with coefficients in $F$ is a homogeneous quadratic polynomial over $F$. 
A \textit{binary form}  $q(x,y) = ax^2 + bxy + cy^2$ where $a,b,c \in F$ is a quadratic form with two variables. 
Any  quadratic form $q$ in $n$ variables over  field $F$ can be written in a diagonal form  $q(x)=a_1 x_1^2 + a_2 x_2^2+ \dots +a_n x_n^2$. 
Such a diagonal form is  denoted by $\langle a_1,\dots,a_n\rangle$; see  \cite[Corollary 2.4]{L2005}. 
The form $f=\langle a_1,\dots ,a_n\rangle$ represents $e\in \dot{F}$ if and only if there exist $ x_1,\dots, x_n $ in $F$ such that $e=a_1 x_1^2 + a_2 x_2^2+ \ldots +a_n x_n^2$. 
The set of all elements of $F$  represented by a form $f$ is denoted $D_f$.
 
 \begin{definition}\cite{MD} The Witt group $W(F)$ is  the abelian group of equivalence classes of non-degenerate symmetric bilinear forms, with   the orthogonal direct sum of forms as the group operation.  The Witt group $W(F)$ with  the tensor product of two bilinear forms as  the ring product  has a commutative ring structure and is called \textit{Witt ring} of $F$.
\end{definition}
 
 The direct product of Witt rings corresponds to the free product of W-groups in the approprative  category; see \cite{MS}. 

\begin{theorem}\cite{MS} 
 Let  $F^{(3)}=F^{(2)}\big(\sqrt{y}: y\in F^{(2)}\big)$ such that $ F^{(2)}(\sqrt{y})/F$ is Galois, $F^{(3)}$ is the compositum  of all extensions  $K$ of $F$ that $Gal(K/F)$ is one of  $\mathbb{Z}/2\mathbb{Z}$, $\mathbb{Z}/4\mathbb{Z}$ or $D_4$.\footnote {This result was recently extended in \cite{EM2011},  although we do not bring the extension in this paper.} 
\end{theorem}

\begin{definition}
\label{d4extension}
 A \textit{$D_4$-extension}  $L$ of  a field $F$ is   a Galois extension of $F$ such that $Gal(L/F)\cong D_4$. 
Let  $a,b \in\dot{F}$ be  independent modulo squares ($ a\dot{F}^2\neq b\dot{F}^2$). 
Then a $D_4^{a,b}$-extension of  a field $F$ is a $D_4$-extension  $K$ of $F$ such that $ F \big(\sqrt{a},\sqrt{b}\big)\subseteq K $ and $Gal \big(K/F(\sqrt{ab})\big)\cong \mathbb{Z}/4\mathbb{Z}$.
\end{definition}

\begin{definition}
 Let $a,b \in \dot{F}$. 
A \textit{quaternion algebra}  $A=(\frac{a,b}{F})$ is the $F$-algebra of dimension four with the basis $\{1,i,j,k\}$, such that $i^2=a$, $j^2=b$,  $k=ij$, $ij=-ji$.
\end{definition}

\begin{lemma}\cite[Theorem.III 2.7]{L2005}\label{splitalgebra}
 We say a quaternion algebra $A=(\frac{a,b}{F})$ is split   over $F$ and, denoted  by $A=(\frac{a,b}{F})=1$  if the binary form $\langle a,b\rangle$ represents 1. 
\end{lemma}
 
\begin{lemma}\cite[Theorem 1.6]{MS1990} 
\label{existextension}
Let $a,b\in F$ be independent modulo squares. Then there exists a $D_4^{a,b}$-extension of $F$ if and only if  quaternion algebra $(\frac{a,b}{F})$ is split.
\end{lemma}

\begin{definition}\cite{MS1990}
An  involution $\sigma$ of $\mathcal{G}_F$ is \textit{simple} if $\sigma \in \Phi_F$ ($\Phi_F$ is  the Frattini  subgroup group of  $\mathcal{G}_F$). 
 An involution   $\sigma$ is \textit{real} if it is not simple and if the fixed field $F^{(3)}_\sigma$ of $\sigma$ is formally real. 
Two involutions $\sigma_1$ and $\sigma_2$ are independent mod $\Phi_F$ if they are in different classes of $\mathcal{G}_F/\Phi_F$. 
When  $\sigma$ is  is not simple and not real, we say $\sigma$ is \textit{nonreal}. 
\end{definition}
 
\begin{theorem}\cite[Theorem 2.7]{MS1990}
  Let  $F$ be a field.
\begin{itemize} 
  \item 
   If  $F$ is  formally real,  then $\mathcal{G}   _F$ contains a real   involution.        
 \item
 Suppose $\mathcal{G}_F$ contains a nonsimple  involution. 
  Then $F$ is formally real.      
   \end{itemize}    
\end{theorem}

\begin{corollary}
If $\sigma$ is a nonsimple involution in $\mathcal{G}_F$ and  $b$ is an element of $\dot{F}$ such that $\sigma \big(\sqrt{b}\big)=-\sqrt{b}$, then: 
           \begin{itemize}
                \item  $b$ is not a sum of two squares,
                \item if $(\frac{b,c}{F})=1$, then $\sigma\big(\sqrt{c}\big)=\sqrt{c}$,
                \item  $\sigma \big(\sqrt{-1}\big)=-\sqrt{-1}$.
           \end{itemize}
\end{corollary}
\begin{proof}
See the proof of \cite[Theorem 2.7]{MS1990}.
\end{proof}

\begin{remark}
Let $\sigma$ be a nonsimple involution in $\mathcal{G}_F$, define  $P_\sigma=\{a \in \dot{F}: \sigma(\sqrt{a})=\sqrt{a}\}.$
It is easy to check  that $P_\sigma$ is an ordering of the field $F$; see \cite[Theorem 2.7]{MS1990}. 
On the other hand for each ordering  $P$ in $X_F$  there exists a real involution $\sigma$ in $\mathcal{G}_F $ such that $P=(F^{(3)}_\sigma)^2\cap \dot{F}$. 
Suppose $\sigma$ and $\tau$ are nonsimple involutions in $\mathcal{G}_F$. 
Then $P_\sigma=P_\tau$ if and only if $\sigma \Phi_F=\tau\Phi_F$; see \cite[Proposition 2.9]{MS1990}. 
Thus $\sigma=\sigma_P$  to the mod $ \Phi_F$.
\end{remark}

\begin{theorem}\cite[Corollary 2.10]{MS1990}
\label{bijectionsigmapsigma}
If $F$ is a Pythagorean formally real field, then there exists a bijection between the set of  nontrivial cosets $\sigma\Phi_F$  of   $\mathcal{G}_F$ $(\sigma$ is an involution) and the set of orderings of  the field $F$. 
The bijection is given by: $\sigma\Phi_F\leftrightarrow P_\sigma. $
\end{theorem}

\begin{theorem}\cite[Theorem 2.11]{MS1990}\label{fratinisugroup} 
Let  $F$ be a  formally real field. 
Then the following conditions are equivalent.
  \begin{enumerate}
   \item
       $F$ is Pythagorean.
    \item
      $\mathcal{G}_F$ is generated by involutions.
    \item
       $\Phi_F=[\mathcal{G}_F,\mathcal{G}_F]$.
   \end{enumerate}
\end{theorem}

\begin{remark}\label{generatorofwgroup}
The involutions that generate the group $\mathcal{G}_F$  in the last theorem are real; see \cite[Theorem 2.11]{MS1990}. 
On the other hand, by Theorem \ref{bijectionsigmapsigma} there is a bijection  between the set of nontrivial coset $\sigma\Phi_F$ and the set of orderings of the field $F$, so the group $\mathcal{G}_F$  is generated by the space of orderings of the field $F$. 
Indeed the group $\mathcal{G}_F$ is a pro-2-group, which is a topological group with the Krull topology, and belongs to the category $\mathcal{C}$ of \cite{MS1990}. 
Commutators of the group $\mathcal{G}_F$ are in the center, the Frattini subgroup $\Phi_F$ is equal to $Gal \big(F^{(3)}/F^{(2)}\big)$, and the commutator subgroup $[\mathcal{G}_F, \mathcal{G}_F]$ is a subset of the Frattini subgroup $\Phi_F$; see \cite{MS1990}. 
Now for pro-2-group $G$  define 2-descending central sequence of $G$ by $G^{(1)}=G$ and  $G^{(i+1)}=\big(G^{(i)}\big)^2 [G^{(i)},G]$ for $i=1,2,\dots$. 
Then $\mathcal{C}$ would be the full category of all pro-2-groups $G$ with  $ G^{(3)}=\{1\}$; see \cite{MS1996}. 
\end{remark}
  
\begin{lemma}\cite[Proposition 2.1]{MS1996}\label{G_F}
If $G_F=Gal\big(F(2)/F\big)$ and  $G_F^{(3)}=G_F^4[G_F^2,G_F]$ is the third term of the 2-descending central sequence for $G_F$, then  $\mathcal{G}_F{\cong} G_F/G_F^{(3)}$. 
\end{lemma}

\begin{lemma}\label{categoryC} \cite[\S1.1]{MS}
The category $\mathcal{C}$ is the full subcategory of the category of pro-2-groups whose objects are those pro-2-groups $G$ satisfying:
  \begin{itemize}
   \item $g^4=1,\,\forall g \in G$,
    \item $g^2\in Z(G)$ where $Z(G)$ is the center of $G,\,\,\forall g \in G$.
     \end{itemize}
\end{lemma}

\begin{definition}\cite{B,W1981} Let $F$ be a filed. An element  $a$ in $\dot{F}$ is called \textit{rigid} if 
$D\langle 1,a\rangle=\dot{F^2}\cup a\dot{F^2}$. 
An element  $a \in  \dot{F}$ is called double-rigid  if both $a$ and $-a$ are rigid. 
Define 
\[
Bas(F)=\{a\in \dot{F}\, | \, a\;\hbox{is not double-rigid}\}\cup\dot{F^2}\cup -\dot{F^2},
\]
which  is a  subgroup of $\dot{F}$. 
\end{definition}

If $F$ is a Pythagorean formally real field and  $|\dot{F}/\dot{F^2}|>2$, then $-1$  is not rigid; see \cite{BCW}. 
~So Bas$(F)=\{a\in \dot{F}\, | \, a\;\hbox{is  not double-rigid}\}$ in this case. Following the notation of \cite{MS}, let 
$H=\{\sigma\in \mathcal{G}_F \, | \, \sigma\big(\sqrt{-1}\big)=\sqrt{-1}\}$. We can choose a basis $\{-1,a_i\, |\,  i \in I\}$ for $\dot{F}/\dot{F}^2$ such that $\{-1,a_i \, | \,  i \in I^{'}\}$ is a basis for the basic part $\hbox{Bas}(F)/\dot{F}^2$ of $\dot{F}/\dot{F}^2$, and $I'\subset I$. 
Let $J=I \backslash I^{'}$. Then $\{a_j \, |\, j \in J\}$ is the set of all  double-rigid elements of the basis $\{-1,a_i\, |\,  i \in I\}$. 
The set $\{ \sigma_{-1},\sigma_i \, | \, i \in I\}$ is a  dual set corresponding to the  basis $\{-1,a_i \, | \, i \in I\}$ which  means $ \sigma_i\big(\sqrt{a_j}\big)=(-1)^{\delta^{i,j}}\sqrt{a_j}$
where $\delta^{i,j}=1$ if $i=j=1$  and $\delta^{i,j}=0$ otherwise.

\begin{theorem}\label{Z(H)} Let $F$ be a Pythagorean formally real field such that $|\dot{F}/\dot{F^2}|>2$, $\Delta_J$  denotes the  subgroup of $\mathcal{G}_F$ generated  by  $\{\sigma_j \, | \,j \in J\}$, and  $Z(H)$ is the center of $H$. 
Then 
\[
Z(H)=\{\sigma \in \mathcal{G}_F \,| \, \sigma \big(\sqrt{b}\big)=\sqrt{b}\;\forall b \in \hbox{Bas}(F)\}=\Delta_J.
\]
Also  any $\sigma_j$ has order four and $\Delta_{J}\cong\prod_{J}\mathbb{Z}/4\mathbb{Z}$. 
\end{theorem}

\begin{proof}
If $F$ is a Pythagorean formally real field, then $D_F\langle 1,1 \rangle=\dot{F^2}$, so the  left side equality is a direct result of \cite[Theorem 4.2]{MS}.
The definition of $\Delta_J$ gives us the  right side equality. 
For the third statement see \cite[Corollary 3.3]{MS}.
\end{proof}

 \begin{theorem}\cite[Theorem 3.5]{MS}\label{deltaJ}
Let $F$ be a field. 
Then $\mathcal{G}_F\cong\Delta_{J}\rtimes \mathcal{G}_K=(\prod_{J}\mathbb{Z}/4\mathbb{Z})\rtimes \mathcal{G}_K$, where  $\mathcal{G}_K$ is generated by $\{\sigma_i \,|\, i\in I\setminus J\}\cup\sigma_{-1}$. 
The group $\mathcal{G}_K$ acts on $\Delta_{J}$ by:
 \[
  \sigma_i^{-1} \tau\sigma_i=\tau,
\qquad
 \sigma_{-1}^{-1} \tau\sigma_{-1}=\tau ^3 \quad\tau \in \Delta_{J},\, \sigma_i \in \mathcal{G}_K.
 \]   
\end{theorem}

\begin{remark}\label{fieldK}
In the last theorem $\mathcal{G}_K$ is the W-group of some suitable  field $K$. 
Based on the work;  \cite{M1980, W1981, JW, AEJ},  Min\'a\v{c} and Smith proved the last theorem (which is in the general case). 
In the case that $F$ is a Pythagorean formally real field, $K$ is  the residue field $\overline{F}$ of some 2-Henselian valuation on $F$; see \cite{MS}. 
\end{remark}

\begin{lemma}\cite[Lemma 3.14]{L1983}\label{2-henselian}
Let $\nu$ be  a valuation on $F$ with the maximal ideal $\mathfrak{m}$ and  residue field $\overline{F}$ such that char$\overline{F}\neq 2$. Then the valuation $\nu$ on $F$  is 2-Henselian  if and only if  $1+\mathfrak{m}\subseteq F^2$.
\end{lemma}

\begin{theorem}\cite[Theorem 3.16]{L1983}\label{residuefield}
If $\nu$  is a  2-Henselian valuation  on a field $F,$ then
  \begin{itemize}
    \item   $F$ is   formally real if and only  if,  $\overline{F}$ is         
        formally real,
    \item     $F$ is   Pythagorean   if and only  if, $\overline{F}$ is Pythagorean.
    \end{itemize}
\end{theorem}
\section{Space of orderings  of a pythagorean formally real fields }
\subsection{Set of orderings of the  field $F$}
Let $F$ be  a Pythagorean formally real field, $X_F$  the set of all  its orderings, $G=\dot{F}/ \dot{F^2}$,  $\mathcal{G}_F$  the  W-group of $F$ and let $\mathcal{X}_F =\{\sigma\Phi_F\, | \, \sigma^2=1,\, \sigma \notin \Phi_F\}$ be the set of  classes of non simple  involutions of $\mathcal{G}_F$ modulo  the   Frattini subgroup of   $\mathcal{G}_F$. 
For any $\sigma$ in $\mathcal{X}_F$ consider $P_\sigma=\{a \in \dot{F}\,|\, \sigma \big(\sqrt{a}\big)=\sqrt{a}\}$. 
By  theorem \ref{bijectionsigmapsigma} we can identify  $X_F$,the set of orderings of $F$, with  $\mathcal{X}_F$. 
Now for any involution $\sigma $  in $\mathcal{G}_F$ and  $b$ in $\dot{F}/\dot{F}^{2}$, let $sgn\,\sigma(b)=\sigma \big(\sqrt{b}\big)/\sqrt{b}$, so $sgn\,\sigma(b)\in\{\pm 1\}$ and $sgn\, \sigma$ is an element of $\chi\big(\dot{F}/\dot{F}^2\big)$. 
If sgn$P_\sigma$ be the classical signature function,
 \begin{equation*}
 \hbox{sgn} P_{\sigma}(b)=
  \begin{cases}
   1 & \text{if }  b\in P_\sigma\\
   -1 & \text{if } b\notin P_\sigma.\\
   \end{cases}
\end{equation*} 
Then $sgn\,\sigma(b)=\hbox{sgn} P_{\sigma}(b)$ for any $b$ in $ \dot{F}/\dot{F}^2$. 
On the other hand any element $\gamma$ in $\chi\big(\dot{F}/\dot{F}^2\big)$ induces  $\overline{\gamma}: F^{(2)}\rightarrow F^{(2)}$ which fixes $F$, and  $\overline{\gamma}$ is an involution in $Gal\big(F^{(2)}/F\big)$. 
We can extend $\overline{\gamma}$ to an involution in $Gal\big(F^{(3)}/F\big)$, so the set of involutions of $Gal\big(F^{(3)}/F\big)$ corresponds to  
$\chi\big(\dot{F}/\dot{F}^2\big)$.

 \begin{definition}\cite[\S 1]{L1983}
For any ordering $P \in X_F$ and form $f=\langle a_1,\dots ,a_n\rangle$, define the $P$-signature of $f$ by $\hbox{sgn}_P(f)=\sum_1^n\hbox{sgn}\, P(a_i)$. 
Two forms $f=\langle a_1,\dots ,a_n\rangle$ and $g=\langle b_1,\dots ,b_m\rangle$ are isometric if and only if  they have the same dimension and  signature with respect to all  ordering $P$ in $X_F$; see \cite[Definition 1.11]{L1983}.
 \end{definition}

\begin{theorem}\label{xFspace of oredring}
Let $F$ be  a Pythagorean formally real field, then $\big(X_F, \dot{F}/\dot{F^2}\big)$ is a space of orderings.
\end{theorem}

\begin{proof}
By Definition \ref{orderingspace}, $\big(X_F, \dot{F}/\dot{F^2}\big)$ is a space of orderings, if it satisfies  the following properties.
    \begin{enumerate}
      \item $X_F$ is a closed subset of  the set of involutions of $        \mathcal{G}_F$.
      \item $\sigma\big(\sqrt{-1}\big)=-\sqrt{-1}$.
      \item $X_F^{\perp}=\{a \in \dot{F}/\dot{F}^{2} \, | \, \sigma\big(\sqrt{a}\big)=\sqrt{a} \, \hbox{ for  any}\, \sigma \in X \}$ 
       is the trivial preordering $\dot{F}^2$. 
      \item If $f$ and $g$ are  two forms over $\dot{F}/\dot{F^2}$ and        if $x\in D_{f\oplus g}$, then there exist $y\in D_f$ and $ z \in D_g$  such that $x\in D_{\langle y,z\rangle}$. 
   \end{enumerate}
  Property (1): Assume that $ K/F$  is any Galois extension, which may be possibly infinite.  
We recall the Krull topology on  $Gal(K/F)$ which is defined via a neighbourhood basis $\mathcal{U}(K/F)=\{Gal(K/N)\,|\,N\in \mathcal{N}\}$ at the unit element, where $\mathcal{N}$ is the set of all subfields $N$  of $K$ which contain $F$ and  are finite and  Galois over $F$.
Now we consider $K=F^{(3)}$. 
First, we will show $\Phi\big(Gal(F^{(3)}/F)\big)$, the Frattini subgroup of $Gal\big(F^{(3)}/F\big)$,  is a closed subgroup of $Gal\big(F^{(3)}/F\big)$. 
By definition of the Frattini subgroup,  $\Phi\big(Gal(F^{(3)}/F)\big)=\bigcap H_i$  for all maximal subgroup $H_i$ of $\mathcal{G}_F$. 
If $\tau \in  \Phi\big(Gal(F^{(3)}/F)\big)$ and $\tau\notin Gal\big(F^{(3)}/F^{(2)}\big)$, there exist $b\in \dot{F}$ such that $\tau(\sqrt{b})=-\sqrt{b}$, so $\tau\notin Gal\big(F^{(3)}/F(\sqrt{b})\big)$.
 But $Gal\big(F^{(3)}/F(\sqrt{b})\big)$ has index 2 in $\mathcal{G}_F$, so it  is a maximal subgroup $\mathcal{G}_F$ and this is contradiction. 
Therefore $\Phi\big(Gal(F^{(3)}/F)\big)$ contains  all $\tau\in Gal\big(F^{(3)}/F\big)$ which fix $F^{(2)}$. 
Thus  if $\gamma\notin\Phi\big(Gal(F^{(3)}/F)\big)$, then there exists $a\in \dot{F}$ such that $ \gamma\big(\sqrt{a}\big)=-\sqrt{a}$. 
Any $\sigma$  in $\gamma\cdot Gal\big(F^{(3)}/F(\sqrt{a})\big)$ has the form  $\sigma= \gamma \alpha$, $\alpha \in Gal\big(F^{(3)}/F(\sqrt{a})\big)$, so  $\sigma\big(\sqrt{a}\big)=-\sqrt{a}$. 
As $ \gamma \cdot Gal\big(F^{(3)}/F(\sqrt{a})\big)$ is an open neighborhood of $\gamma$,  we see that $\Phi\big(Gal(F^{(3)}/F)\big)$ is a closed subgroup of $Gal\big(F^{(3)}/F\big)$. 
Now we can consider the quotient topology on $W=Gal\big(F^{(3)}/F)/\Phi(Gal(F^{(3)}/F)\big)$. 
We consider $X_F$ to be a subset of $W,$ and we will show that $X_F$ is  a closed subset of $W$.\\

Let $[\alpha] $ be an element in $W\setminus X_F$ while $[\alpha]=\alpha\Phi\big(Gal(F^{(3)}/F)\big)$. 
Then either $\alpha$ in $\Phi\big(Gal(F^{(3)}/F)\big)$, or $\alpha$ is not an involution of $Gal\big(F^{(3)}/F\big)$ (this follows from Theorem \ref{fratinisugroup}). 
If  $\alpha\in\Phi\big(Gal(F^{(3)}/F)\big)$, then $\alpha\big(\sqrt{-1}\big)=\sqrt{-1}$and $\alpha Gal\big(F^{(3)}/F(\sqrt{-1})\big)$ has image in $W$ disjoint with $X_F$. 
If $\alpha$ is not involution, then by Theorem \ref{bijectionsigmapsigma},  $P_\alpha=\{a\in \dot{F}/\dot{F}^2\,|\,\, \alpha\big(\sqrt{a}\big) =\sqrt{a}\}$ is not an  ordering of $F$. 
As  $P_\alpha$ is not an ordering we see that  $P_\alpha$ is not additively closed, so  there exist  $h,q \in P_\alpha$ such that  $h+q \notin P_\alpha$, or equivalently $\alpha \big(\sqrt{h}\big) = \sqrt{h}$, $\alpha (\sqrt{q}) = \sqrt{q}$ and  $\alpha (\sqrt{q{+}h}) = \sqrt{q{+}h}$. 
Let  $H=Gal(F^{(3)}/K)$, where $K=F\big(\sqrt{h},\sqrt{q},\sqrt{h{+}q}\big)$. 
Then the action of all elements $\alpha H$ on $K$ is the same as the action of $\alpha$ on $K$. 
This means that for $\beta\in \alpha H$ and 
$P_\beta=\{a\in \dot{F}/\dot{F}^2\,|\; \beta\big(\sqrt{a}\big)=\sqrt{a}\}$, we have $h,q \in P_\beta$ but  $h+q \notin P_\beta$, so $P_\beta$  is not an ordering of $F$. 
Hence the image of the open neighborhood  $\alpha H$ of $\alpha$ in $W$ has an empty intersection with $X_F$, and $X_F$ is a  closed subset of $W$.\\

For  property (2), see \cite[Theorem 2.7]{MS1990}. 
Let $X_F = \{\sigma_1\Phi_F,\sigma_2\Phi_F,\dots,\sigma_n\Phi_F\}$ by Theorem \ref{bijectionsigmapsigma} any $\sigma_i\Phi_F$  corresponds to the ordering $P_i$ of $F$. 
Now by Artin's theorem $X^{\perp}=\bigcap_{ P_i \in X_F}P_i  =\dot{F}^2$; see \cite[Theorem 1.6]{L1983}. 
For property (4), let $f=\langle a_1,\dots,a_m\rangle$ and $ g=\langle b_1,\dots,b_n\rangle$ be two forms on $F$. Suppose   there exist  elements $f_1,\dots,f_{m+n}\in F$ such that $x=a_1f_1^2+, \dots, +a_{m+n}f_{m+n}^2$. 
Let 
 \[
 y=a_1f_1^2+, \dots, +a_n f_n^2, \qquad z=a_{n+1}f_{n+1}^2+, \dots ,+a_{m+n}f_{m+n}^2
 \]  
which means there exist $y\in D_f$ and $ z\in D_g$ such that $x\in D_{\langle y,z\rangle}$.
\end{proof}

\begin{remark}
If $\{P_{\sigma_1},P_{\sigma_2},\dots ,P_{\sigma_m}\}$ be an arbitrary subset of a space of orderings $X_F$, $Y=\{P \in X_F \,|\,P=P_{\sigma_{i_1}},\dots ,P_{\sigma_{i_k}}\}$ and $T=\bigcap P_{\sigma_i}$, $ i=1,\dots,m$. 
If  $Y$, $T$ satisfy the  duality condition $T=\bigcap P_\sigma,\, P_\sigma\in Y$ and $ Y=\{P \in X_F\,|\, T\subseteq P\}$, then $(Y,\dot{F}/\dot{T})$ is  a subspace of $\big(X_F,\dot{F}/\dot{F}^2\big)$ and there is  a Pythagorean formally real field $E$ such that $Y$  corresponds to  $X_E$; see \cite[Theorems 4.10, 2.2]{M1979}.  
\end{remark}
\subsection{Main Theorem}
  The following Theorem  corresponds to the  Basic Lemma of \cite{M1979}, which was critical  in the classification of general spaces of orderings.

\begin{theorem}
\label{maintheorem}
There does not exist any Pythagorean formally real field $F$ such that
$ \mathcal{G}_F=\langle\sigma_1,\dots,\sigma_5\, | \, \sigma_1^2,\dots, \sigma_5^2=(\sigma_1\cdots\sigma_5)^2=1,\,i=1,\dots,5\rangle$ 
and $\sigma_1,\dots,\sigma_5$ are independent involutions  mod $\Phi_F$.
\end{theorem}

\begin{proof}
Since involutions $\sigma_1,\dots,\sigma_5$ are independent modulo $\Phi_F$, there exists a dual basis $\{a_1,\dots,a_5\}$ such that $\sigma_i\big(\sqrt{a_j}\big)=(-1)^{\delta_{i,j}}\big(\sqrt{a_j}\big)$.\\

Claim 1: If there exists a $D_4^{\frac{a_ia_j}{c},\frac{1}{c}}$ extension $K$ of $F$ such that $i,j\in \{1, \dots, 5\}$ and  
$c =a_1^{\epsilon_1}\dots a_5^{\epsilon_5}$, $\epsilon_i \in \{0,1\}$, 
then $c=1$ or $c=a_ia_j$. 
Indeed, if we consider the following diagram of a $D_4^{\frac{a_ia_j}{c},\frac{1}{c}}$-extension $K$ of  $F$, without loss of generality, we can assume $\frac{a_ia_j}{c}=\alpha$, $\frac{1}{c}=\beta$.  
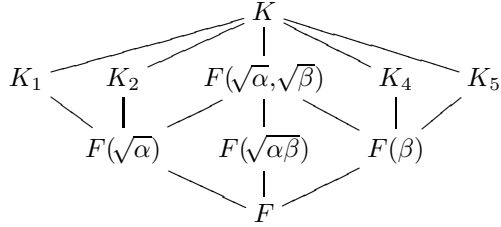
\begin{equation*}
\xymatrix@C=1em@R=1em{
 &     &    K \ar@{-}[lld] \ar@{-}[ld] \ar@{-}[d] \ar@{-}[dr]  \ar@{-}[drr] &   &      \\
K_1 \ar@{-}[rd] & K_2 \ar@{-}[d] &    F(\sqrt\alpha,\sqrt\beta)\ar@{-}[ld] \ar@{-}[d] \ar@{-}[rd]    & K_4 \ar@{-}[d] & K_5 \ar@{-}[ld] \\
 &   F(\sqrt\alpha)\ar@{-}[rd] & F(\sqrt{\alpha\beta}) \ar@{-}[d] & F(\beta)\ar@{-}[ld] &  \\
& & F & &}
\end{equation*}
As $c =a_1^{\epsilon_1}\dots a_5^{\epsilon_5}$, $\epsilon_i \in \{0,1\}$, we need to show that $\epsilon_t=0$ for $ t\neq i,j$.  
Assume if $ t\neq i,j$, then $\epsilon_t=1 $. 
Let $\overline{\sigma_t}$ be  the restriction of $\sigma_t$ to the  $D_4^{\alpha,\beta}$-extension $K$ of $ F$, and let $K_{\sigma_t}$ be the fixed field of  $\overline{\sigma_t}$. 
We have ${\overline\sigma_t}\big(\sqrt{\beta}\big)={-}\sqrt{\beta}$ 
and $\overline{\sigma_t}\big(\sqrt{\alpha}\big)={-}\sqrt{\alpha}$, 
so $\sqrt{\alpha},\sqrt{\beta} \notin K_{\sigma_t}$. 
All  five intermediate fields of index two in the above diagram contain  $\sqrt{\alpha}$ or $ \sqrt{\beta}$. 
So $K_{\sigma_t}$ can not be any of the five intermediate fields of index two in this diagram. 
Thus it should be $F\big(\sqrt{\alpha\beta}\big)$. 
On the other hand $Gal(K/F(\sqrt{\alpha\beta}))\cong C_4$, but  $\langle\sigma_t\rangle \cong C_2$, and this is  a contradiction. 
If  $\epsilon_i=1$,  $\epsilon_j=0$ let 
$\tau=\overline{\sigma_1\sigma_2\sigma_3\sigma_4\sigma_5}$ be the restriction of $\sigma_1\sigma_2\sigma_3\sigma_4\sigma_5$ to the field $K$. 
We have $\tau\big(\sqrt{\alpha}\big)=-\sqrt{\alpha}$ and 
$\tau\big(\sqrt{\beta}\big)=-\sqrt{\beta}$. 
If $K_\tau$ is the fixed field of $\tau$, then $\sqrt{\alpha}, \sqrt{\beta}\notin K_\tau$, so $K_\tau$ can not be any of  the five intermediate fields of index two in the above diagram. 
But $\tau^2=(\sigma_1\sigma_2\sigma_3\sigma_4\sigma_5)^2 \cong C_2$,  which is   contradiction. 
In the same way  if $\epsilon_i=0$, $\epsilon_j=1$ 
we  get a contradiction, so $c=a_ia_j$ or $c=1$.\\

Let $a=a_1a_2a_3a_4a_5$, $f=\langle 1,aa_5, aa_4\rangle$, and   
$g=\langle aa_3, a_1 a_3, a_2 a_3\rangle$.  
Then
\begin{align*}
 \sigma_1(f){=}\sigma_1(1)+\sigma_1(aa_5)+\sigma_1(aa_4)=-1,\\ 
 \sigma_1(g)=\sigma_1(aa_3)+\sigma_1(a_1a_3)+\sigma_1(a_2a_3)=-1.
\end{align*} 
So  $\sigma_1(f)=\sigma_1(g)$, any ordering $\sigma$ in $X_F$ is a product of an odd number of $\sigma_1, \dots, \sigma_5$, and  it is easy to check that $\sigma(f)=\sigma(g)$. 
Therefore $\hbox{sgn}(f)=\hbox{sgn}(g) $, so $f{\cong} g$ and $ aa_3 \in D_f$, which means there exist $w,y,z \in F $ such that  $w^2+ aa_5 y^2 + aa_4 z^2= aa_3$.
So $w^2+ (aa_5 \frac{y^2}{z^2} + aa_4)z^2= aa_3$. 
Let $b=aa_5 \frac{y^2}{z^2} + aa_4$, then $\frac{aa_5}{b} \frac{y^2}{z^2} +\frac{ aa_4}{b}=1$ and $bz^2+w^2 =aa_3$. 
If $\lambda=\frac{aa_5}{b}, \gamma=\frac{aa_4}{b}$, by Theorem \ref{existextension} there exists a $D_4^{\gamma,\lambda}$-extension $K$ of $ F$. 
Let $a_i=a_4$, $a_j=a_5$ in  the diagram of claim 1, then we  consider 
$\gamma=\frac{aa_4}{b}=1/c$ and  $\lambda=\frac{aa_5}{b}\frac{aa_4\times a_4a_5}{b}=\frac{a_4a_5}{c}$. 
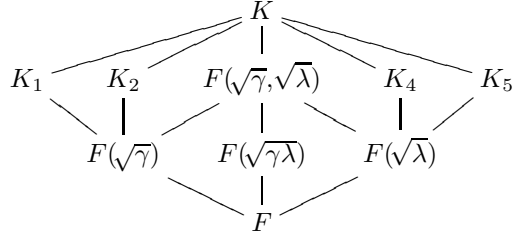
\begin{equation*}
    \xymatrix@C=1em@R=1em{
     &     &    K \ar@{-}[lld] \ar@{-}[ld] \ar@{-}[d] \ar@{-}[dr] \ar@{-}[drr] &   &      \\
K_1 \ar@{-}[rd] & K_2 \ar@{-}[d]&    F(\sqrt{\gamma},\sqrt\lambda)\ar@{-}[ld] \ar@{-}[d] \ar@{-}[rd]    & K_4 \ar@{-}[d] & K_5 \ar@{-}[ld] \\
 &   F(\sqrt\gamma)\ar@{-}[rd] & F(\sqrt{\gamma\lambda}) \ar@{-}[d] & F(\sqrt \lambda)\ar@{-}[ld] &  \\
& & F & &}  
\end{equation*}
But in   claim 1 we proved $c=1$ or $ c=a_4a_5$. 
If $c=1$, then $\frac{aa_4}{b}=1$ and $b=aa_4$, in the case $c=a_4a_5$, $\frac{aa_5}{b}=1$ and $b=aa_5$. 
On the other hand,  because of $bz^2+w^2=aa_3$, we have $\frac{b}{aa_3}z^2 +\frac{1}{ aa_3}w^2=1$. 
So there exists a $D_4^{baa_3,aa_3}$-extension $K$ of $F$. 
Now consider the following diagram of this extension:
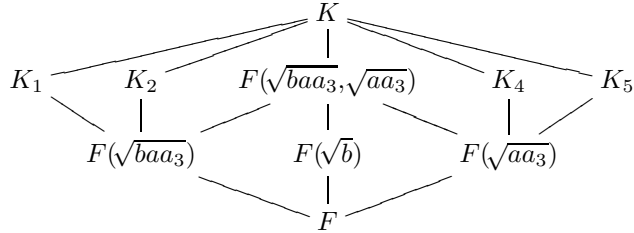
\begin{equation*}
    \xymatrix@C=1em@R=1em{
     &     &    K \ar@{-}[lld] \ar@{-}[ld] \ar@{-}[d] \ar@{-}[dr] \ar@{-}[drr] &   &      \\
K_1 \ar@{-}[rd] & K_2 \ar@{-}[d]&    F(\sqrt{baa_3},\sqrt {aa_3})\ar@{-}[ld] \ar@{-}[d] \ar@{-}[rd]    & K_4 \ar@{-}[d] & K_5 \ar@{-}[ld] \\
 &   F(\sqrt{baa_3})\ar@{-}[rd] & F(\sqrt{b}) \ar@{-}[d] & F(\sqrt {aa_3})\ar@{-}[ld] &  \\
& & F & &}
\end{equation*}

If $b=aa_5$, let $\overline{\sigma_5}$  be the restriction of $\sigma_5$ to the extension $K$, and let $K_{\sigma_5}$ be the fixed field of $\overline{\sigma_5}$. 
Then $\overline{\sigma_5}\big(\sqrt{baa_3}\big)=\overline{\sigma_5}\big(\sqrt{a_3a_5}\big)=-\sqrt{a_3a_5}$ and 
$\overline{\sigma_5}\big(\sqrt{aa_3}\big)=-\sqrt{aa_3}$, so 
$\sqrt{baa_3},\sqrt{aa_3}\notin K_{\sigma_5}$.  
Therefore, $K_{\sigma_5}$ cannot be any of the  five intermediate fields of index two in this diagram. 
Therefore  $K_{\sigma_5}$  does not equal  $F\big(\sqrt{b}\big)$, but  $Gal\big(K/F(\sqrt{b})\big)\cong C_4$, whereas  $\langle\sigma_5\rangle\cong C_2$; this is contradiction. 
In the  case of  $b=a a_4$, and for $\sigma_4$ the argument is the same.  
So the  group 
\[
\langle\sigma_1,\dots,\sigma_5 \, | \,(\sigma_i)^2=(\sigma_1\cdots\sigma_5)^2=1\,\,i=1,\dots,5\rangle
\]   
cannot be the W-group of any Pythagorean  formally real field $F$.
\end{proof}

\begin{lemma}
Let $\{\sigma_1,\dots,\sigma_n\}$ be a basis for $X_F$, and let $\tau \in \mathcal{G}_F$ be such that $\sigma_i\tau\in X_F$ for $ i=1\dots n$. 
Then $\tau X_F=X_F$.
\end{lemma}

\begin{proof}
This is a direct  result of  Theorem \ref{xFspace of oredring} and \cite[Lemma 4.2]{M1979}.
\end{proof}

\begin{remark}  
In the last Lemma, $\tau$ is an element of $\mathcal{G}_F$ such that 
$\sigma_i\tau\in X_F$. 
By Definition \ref{orderingspace}  part two, 
$\sigma_i\tau\big(\sqrt{-1}\big)=-\sqrt{-1}$ and  $\sigma_i\big(\sqrt{-1}\big)=-\sqrt{-1}$. 
Therefore $\tau\big(\sqrt{-1}\big)=\sqrt{-1}$  and $\tau$ is an element of $Gal\big(F^{(3)}/F(\sqrt{-1})\big)$. 
The  translation group $T$  for the general space of ordering is the set of all  $\tau $ in $\chi(G)$  such that $\tau X=X$, so for the  space of orderings $\big(X_F, \dot{F}/ \dot{F^2}\big)$, the translation group $T$ is  the set of all $\tau $ in $Gal\big(F^{(3)}/F(\sqrt{-1})\big)$  such that $\tau X_F=X_F$. 
But $Gal\big(F^{(3)}/F(\sqrt{-1})\big)$ is a subgroup of $\mathcal{G}_F$ which is generated by orderings of the field $F$  Remark \ref{generatorofwgroup}. 
So the translation group $T$ is  exactly the set of all elements $\tau$ 
of group
 \[
 Gal\big(F^{(3)}/F(\sqrt{-1})\big)=\langle \sigma_i\sigma_j\, | \, \sigma_i,\sigma_j\in X_F \rangle
 \]
which have the  property $\tau X_F=X_F$. 
In Theorem \ref{nontrivialcenter} we show that  $\tau$ is an element of  the center of $Gal\big(F^{(3)}/F(\sqrt{-1})\big)$.
\end{remark}

\begin{lemma}\label{thereexistgamma}
Let $X_\alpha = \{\sigma\in X_F\, | \, \sigma\alpha\in X_F\}$ and     
$\alpha,\beta\neq 1 $ are elements of $Gal\big(F^{(3)}/F(\sqrt{-1})\big)$. 
If $X_\alpha\cap X_\beta\neq\emptyset$ and $\hbox{rank}\,X_\alpha$ and  $\hbox{rank}\,X_\beta\geq3$, 
then  there exist $\gamma$ in $Gal\big(F^{(3)}/F(\sqrt{-1})\big)$ 
such that $X_\alpha,X_\beta\subseteq X_\gamma$.
\end{lemma}

\begin{proof}
This is a direct result of  \cite[Lemma 4.6]{M1979} for  $\big(X_F,\dot{F}/ \dot{F^2}\big)$.
\end{proof}

\begin{theorem}\label{existtau}
If $X_F$ is  a connected  space of orderings  of a field $F$ and rank $X_F\neq1$, then there exists $\tau\in Gal\big(F^{(3)}/F(\sqrt{-1})\big)$, $\tau \neq1$ such that $ \tau X_F=X_F$.
\end{theorem}
\begin{proof}
Apply  \cite[Theorem 4.7]{M1979} to the space of orderings $\big(X_F, \dot{F}/ \dot{F^2}\big)$.
\end{proof}

\begin{theorem}\label{nontrivialcenter}
Let $X_F$ be a connected space of orderings of a  Pythagorean formally real  field $F$. 
Then $Gal\big(F^{(3)}/F(\sqrt{-1})\big)$ has a nontrivial center.
\end{theorem}

\begin{proof}
Let $|X_F|=n$. By Theorem  \ref{existtau} there exists an element 
$\tau=\sigma_1\sigma_2$ in the group $Gal\big(F^{(3)}/F(\sqrt{-1})\big)$ 
such that $ \tau X_F=X_F$. 
Therefore for any $\sigma_i\in X_F$, $\tau\sigma_i=\sigma_{k(i)},\, k(i)\in\{1,\dots,n\}$. 
So there exist $m,n$ such that $\tau\sigma_i=\sigma_m$ and  $\tau\sigma_j=\sigma_n$. 
Therefore 
\[
\sigma_1\sigma_2\sigma_i=\sigma_m\Longrightarrow \sigma_i=\sigma_i^{-1}=\sigma_2\sigma_1\sigma_m\qquad
\sigma_1\sigma_2\sigma_j=\sigma_n \Longrightarrow \sigma_j=\sigma_j^{-1}=\sigma_2\sigma_1\sigma_n. 
 \] 
For any element $\gamma=\sigma_i\sigma_j$ in $Gal(F^{(3)}/F(\sqrt{-1}))$ we would have
\begin {align*}
\tau\sigma_i\sigma_j\tau^{-1}\sigma_j\sigma_i 
&=\sigma_m\sigma_n\sigma_j\sigma_i=\sigma_m\sigma_n\sigma_2\sigma_1\sigma_n\sigma_2\sigma_1\sigma_m=\sigma_m\sigma_n(\sigma_2\sigma_1\sigma_n)^2\sigma_n\sigma_m\\
(\sigma_j =\sigma_2\sigma_1\sigma_n)^{2}&=1
\Longrightarrow \sigma_m\sigma_n(\sigma_2\sigma_1\sigma_n)^2\sigma_n\sigma_m =1
\Longrightarrow\tau\sigma_i\sigma_j=\sigma_i\sigma_j\tau.
\end{align*}
So $\tau$ is an element of the  center of $Gal\big(F^{(3)}/F(\sqrt{-1})\big)$.
\end{proof}

The next theorem is a generalization of \cite[Theorem 3.5]{MS}.

\begin{theorem}\label{Connected case}
If $F$  is a Pythagorean formally real field with  connected space of  ordering $X_F$, then $\mathcal{G}_F\cong\prod C_{4}\rtimes\mathcal{G}_{\overline{F}}$, where $\overline{F}$ is also a  Pythagorean formally real  field with a smaller number of orderings and a disconnected space of  orderings $ X_{\overline{F}}$. 
For a suitable set of generators $\sigma_i, \sigma_{-1}$, the group  
$\mathcal{G}_{\overline{F}}$ acts on $\prod C_{4}$ by $\sigma_i^{-1} \tau\sigma_i=\tau $ and $\sigma_{-1}^{-1} \tau\sigma_{-1}=\tau ^3$ where, $\tau \in \prod C_{4}$.
\end{theorem}

\begin{proof}
We assume $X_F$ is a connected space, so by Theorem \ref{nontrivialcenter}, $Gal \big( F^{(3)}/F(\sqrt{-1})\big)$ has nontrivial center. 
Any element $\gamma=\sigma_i\sigma_j$ in $Gal \big(F^{(3)}/F(\sqrt{-1}) \big)$, such that $\sigma_i$, $\sigma_j$ are elements  of the basis $\{\sigma_1,\dots,\sigma_n\}$, has order four. 
This is because of $\sigma_i$, $\sigma_j$ are independent involutions  mod $ \Phi_F$, so $\sigma_i\sigma_j\notin \Phi_F$. 
Indeed if $(\sigma_i\sigma_j)^2=1$, then $\sigma_i\sigma_j$ is a non simple involution and $\sigma_i\sigma_j\big(\sqrt{-1}\big)=-\sqrt{-1}$, 
but $\sigma_i\sigma_j\big(\sqrt{-1}\big)=\sigma_i\big(\sigma_j\sqrt{-1}\big)=\sigma_i\big(-\sqrt{-1}\big)=\sqrt{-1}$ which is a contradiction. 
So $\gamma=\sigma_i\sigma_j$ has order four, and the center of 
$Gal \big(F^{(3)}/F(\sqrt{-1})\big)$ is isomorphic to $\prod_1^m C_{4}$. 
Let $H=\{\sigma\in \mathcal{G}_F\, |\,\sigma\big(\sqrt{-1}\big)=\sqrt{-1}\}$. 
By Theorem \ref{Z(H)} 
\[
Z(H)=\{\sigma \in \mathcal{G}_F\,|\, \sigma\big(\sqrt{b}\big)=\sqrt{b}\,\, \forall b \in \hbox{Bas}(F)\}=\Delta_J.
\] 
But $H=Gal \big(F^{(3)}/F(\sqrt{-1}) \big)$, so the center of $Gal \big(F^{(3)}/F(\sqrt{-1}) \big)$ that was nontrivial and isomorphic to  a product of copies of $C_4$  is equal to the set 
\[  
 \{\sigma \in \mathcal{G}_F\,|\, \sigma(\sqrt{b})=\sqrt{b}\,\, \forall b \in \hbox{Bas}(F)\}=\Delta_J.
 \]  
Now let $E= F \big(\sqrt{a_i}\, | \, a_i \in Bas(F) \big)$ be  the fixed field of $Z(H)$ and $X^{'}=\{\sigma\vert_E\, |\,  \sigma\in X_F\}$. 
By Theorem \ref{spaceX'}, $\big(X^{'},\dot{E}/\dot{E^2}\big)$ is a space of orderings. 
Now, by Theorem \ref{existF}, there exists a  Pythagorean formally real field  $L$ such that $\big(X^{'},\dot{E}/\dot{E^2}\big){\sim} \big(X_L,\dot{L}/\dot{L^2}\big)$. 
So $\mathcal{G}_L$,  the  W-group of the Pythagorean  real field $L$, is  generated  by $X_L=X^{'}$. 
The fundamental  theorem of Galois theory implies $\mathcal{G}_F\cong\Delta_J\rtimes Gal(E/F)$. 
But $Gal(E/F)$ is a subgroup of $\mathcal{G}_F$ which is generated by $\{\sigma\vert_E\, |\,  \sigma\in X_F\}=X^{'}$, 
so $\mathcal{G}_F\cong\Delta_J\rtimes \mathcal{G}_L$. 
~On the other hand, by Theorem \ref{deltaJ} and Remark \ref{fieldK}, 
\[
\mathcal{G}_F\cong\Delta_J\rtimes \mathcal{G}_{\overline{F}}\cong\prod C_{4}\rtimes \mathcal{G}_{\overline{F}}
\]
 where $\overline{F}$ is the residue field of some 2-Henselian  valuation on $F$. 
But $Z(H)$ where 
\[
H=Gal \big(F^{(3)}/F(\sqrt{-1})\big)=\langle \sigma_i\sigma_j\,|\, \sigma_i,\sigma_j{\in} X_F\rangle
\] 
is nontrivial. 
If $\tau =\sigma_i\sigma_j$ is an element of $Z(H)$, as $E$ is the  fixed field of $Z(H)$, $\sigma_i\vert_E=\sigma_j\vert_E $ which implies the  rank $X_{\overline{F}}$ is less than the rank $X_F$. 
Also $X_{\overline{F}}$ is  disconnected space of orderings; see  \cite[Remark 1]{M1979}.
 \end{proof}

\begin{remark} 
 Note that in the connected  case, any $\sigma$ in $X_F $ can be written uniquely as a product $t\sigma'$ that $t\in T$, $ \sigma'\in X_{\overline{F}}$. 
Then $|X_F| = |T||X_{\overline{F}}|$ and  $|T| = 2^m$ where $m$ is
the dimension of $T$ as $\mathbb{F}_2$ vector space. 
In Theorem \ref{disconnectedcase} we show that if  the space of orderings $X_F$ is disconnected,  $\mathcal{G}_F$ is equal to the free product of $\mathcal{G}_{F_i}$ in the category $\mathcal{C}$. 
The group  $\mathcal{G}_{F_i}$ is a W-group of  some Pythagorean field $F_i$, with a connected space of orderings $X_{F_i}$.
Recursively, one can apply   Theorem \ref{Connected case} to the Pythagorean fields $F_i$ until the structure of the W-group $\mathcal{G}_F$ is completely determined. 
\end{remark}

\begin{lemma}\cite[\S.2]{Mi1986}\label{G2}
Let $F$ be a Pythagorean formally real field with  space of orderings 
$X_F=\{\sigma_1,\dots,\sigma_n\, |\, \sigma_i^2=1, \, \sigma_i\notin \Phi_F\}$, such that $X_F=X_1\oplus X_2 \oplus\dots\oplus X_k$, where  
$X_i$ are the connected components of $X_F$. 
Then $G_F=G_{F_1}\ast\dots\ast G_{F_k}.$
\end{lemma}

\begin{theorem}\label{disconnectedcase}
If $F$  is a Pythagorean formally real field,  
\[
X_F=\{\sigma_1,\dots,\sigma_n\, |\, \sigma_i^2=1, \,\sigma_i\notin \Phi_F\}\qquad X_F=\bigcup_{i=1}^k \big(X_i,\dot{F}/\dot{F}^2\big).
\] 
The $X_i$ are  the connected components of $X_F$ and the $\mathcal{G}_{F_i}$ are the W-groups corresponding to the  fields $F_i$. 
Then $\mathcal{G}_F\cong \mathcal{G}_{F_1}\ast\dots\ast \mathcal{G}_{F_k}$.  
\end{theorem}

\begin{proof}
For any finite space of orderings $X_i$ there exists a  Pythagorean field $F_i$ such that $X_i$ is equivalent to $X_{F_i}$, the space of orderings of the field $F_i$.\\ 
So $X_F=X_{F_1}\oplus X_{F_2} \oplus\dots\oplus X_{F_k}$. 
By the last lemma, $G_{F}=G_{F_1}\ast\dots\ast G_{F_k}$. 
~On the other hand, $\mathcal{G}_F = G_F/G_F^{(3)}$ where $G_F^{(3)}$ is the third term of 2-descending central sequence  of $G_F$ in Lemma \ref{G_F}. 
By observing that free products in the category of pro-2-groups are mapped to free products in the category $\mathcal{C}$, we obtain our claim (for any field $F$, $G_F/G_F^{(3)}$ is an object of the category $\mathcal{C}$).
\end{proof} 
\subsection{Examples}
 For any  formally real field $F$ the space of orderings  $X_F$ is connected if it has just one connected component. 
By Definition \ref{connectedorderings} the minimum number of orderings in the connected space of orderings is four. 
Let $X_F=\{\sigma_1,\sigma_2,\sigma_3,\sigma_4\}$ be a connected space, where without loss of generality we assume $\sigma_1\sigma_2\sigma_3=\sigma_4$. 
Therefore  $X_F$ has three independent orderings, $|\dot{F}/\dot{F}^2|=2^3$ and $|X_F|=4$.

\begin{example}\label{smallestcnnctedwgroup}
Let $F$ be a Pythagorean field such that $\mathcal{G}_F=\langle\sigma_1,\sigma_2,\sigma_3\rangle$ and $\sigma_{1},\sigma_{1}, \sigma_3$ are independent involutions mod $\Phi_F$. 
Then 
\[
[\sigma_1,\sigma_2][\sigma_2,\sigma_3][\sigma_3,\sigma_1]=1 \Longleftrightarrow \mathcal{G}_F\cong (C_4 \times C_4 )\rtimes C_2.
\]
Suppose $\mathcal{G}_F\cong (C_4 \times C_2 )\rtimes C_2$, $\tau_1$, 
$\tau_2$ are generators of $C_4 \times C_4 $ and $\sigma_1$ the  generator of $C_2$. 
Let $\sigma_{1}\sigma_{2}= \tau_{1}$ and $ \sigma_{1}\sigma_{3}=\tau_2$, so the generator $ \sigma_1$  acts on $\tau_i$ by: 
$\sigma^{-1}_1\tau_i\sigma_1=\tau^{-1}_i$ for $i\in \{1,2\}$ 
which gives  $\tau_i\sigma_1=\sigma_1\tau^3_i$. 
But 
\begin{align*}
[\sigma_1,\sigma_2]&=[\sigma_1,\sigma_1\tau_1] =\sigma^{-1}_1(\sigma_1\tau_1)^{-1}\sigma_1\sigma_1\tau_1
=\tau_1^2\qquad
[\sigma_3,\sigma_1]=[\sigma_1\tau_2,\sigma_2]=\tau^2_2\\ 
[\sigma_2,\sigma_3]&=[\tau_1\sigma_1,\tau_2\sigma_1]=\sigma_1\tau_1^3\sigma_1\tau_2^3\tau_1\sigma_1\tau_2\sigma_1=\tau_1\sigma_1\sigma_1\tau_2^3\tau_1\sigma_1\tau_2\sigma_1=\tau_1\tau_2^3\tau_1\tau_2^3
\end{align*}
Therefore $\big([\sigma_1,\sigma_2][\sigma_2,\sigma_3][\sigma_3,\sigma_1]\big)=\tau^2_1\tau_1\tau_2^3\tau_1\tau_2^3\tau^2_2=1$. 
Conversely, let $F$ be a Pythagorean formally real field generated by independent involutions  $\sigma_1,\sigma_2,\sigma_3$ which satisfy the relation $[\sigma_1,\sigma_2][\sigma_2,\sigma_3][\sigma_3,\sigma_1]=1$. 
Suppose $\tau_1=\sigma_1\sigma_2$ and $\tau_2=\sigma_1\sigma_3$. 
We will show  $\tau_1$ and $ \tau_2$ have order four. 
As $\Phi_F=[\mathcal{G}_F,\mathcal {G}_F]$ and any element in 
$\Phi_F$ has order two, we have 
\[\tau_1^4\sigma_1\sigma_2\sigma_1\sigma_2\sigma_1\sigma_2\sigma_1\sigma_2=[\sigma_1,\sigma_2][\sigma_1,\sigma_2]=1.
\] 
It is easy to check that $\tau_1^2=(\sigma_1\sigma_2)^2\neq1$, so the  order of $\tau_1$ is four. 
In the same way, $\tau_2$ has order four. 
Let $H=\langle\tau_1\rangle\times\langle\tau_2\rangle$ and $ K={}\langle\sigma_1\rangle$, since $\sigma_1^{-1}\tau_1\sigma_1=\tau_1^3$,$\sigma_1^{-1}\tau_2\sigma_1=\tau_2^3$ the group $K$ acts on $H$ by  
$\sigma_1^{-1}\tau_i\sigma_1=\tau_i^{-1}=\tau_i^3$. 
So the  set 
\[
\{(h,k)\,|\, h\in (\langle\tau_1\rangle\times\langle\tau_2\rangle),\,  k\in \langle\sigma_1\rangle\}
\]
has the  group structure $(C_4 \times C_4 )\rtimes C_2$. 
On the other hand $\big(\langle\tau_1\rangle\times\langle\tau_2\rangle\big)\cap  \langle\sigma_1\rangle={1}$ and $\big(\langle\tau_1\rangle\times\langle\tau_2\rangle\big)$ is a normal subgroup, so $\mathcal{G}_F{\cong}(C_4 \times C_4)\rtimes C_2.$
\end{example}

\begin{example}\label{example1}
Let $F$ be a field  such  that; 
\[
X_F=\{\sigma_1,\dots,\sigma_6\,|\,(\sigma_1\sigma_2\sigma_3)^2=1,\,(\sigma_4\sigma_5\sigma_6)^2=1\}.
\] 
We can determine the structure of  W-group $\mathcal{G}_F$  and  the  Frattini subgroup $\Phi_F$.  
The space $X_F$  has two connected components;
\[ 
 X_1=\{\sigma_1,\sigma_2,\sigma_3,\sigma_1\sigma_2\sigma_3\} \qquad X_2=\{\sigma_4,\sigma_5,\sigma_6,\sigma_4\sigma_5\sigma_6\}.
 \]
Let $F_1$ and $F_2$ be the two  Pythagorean formally real fields corresponding to $X_1$ and $X_2$ respectively. 
Their corresponding W-groups are  $ \mathcal{G}_{F_1}$ and $\mathcal{G}_{F_2}$. 
So
\[ 
 \mathcal{G}_{F_1}=\langle\sigma_1,\sigma_2,\sigma_3\,|\,(\sigma_1\sigma_2\sigma_3)^2=1\rangle \qquad \mathcal{G}_{F_2}=\langle\sigma_4,\sigma_5,\sigma_6\,|\,(\sigma_4\sigma_5\sigma_6)^2=1\rangle.
 \]
Apply Theorem \ref{disconnectedcase}  to $F$ and the definition of 
free product of two pro-2-groups, to see that.
\[  
\mathcal{G}_F\cong \mathcal{G}_{F_1}\ast \mathcal{G}_{F_2}\cong\big(\mathcal{G}_{F_1}\times\langle\mathcal{G}_{F_1},\mathcal{G}_{F_2}\rangle\big)\rtimes \mathcal{G}_{F_2}.
\]
As we mentioned in Remark \ref{generatorofwgroup}, commutators are in the center so;
\begin{align*}
 (\sigma_1\sigma_2\sigma_3)^2&= \sigma_1\sigma_2\sigma_3 \sigma_1\sigma_2\sigma_3=\sigma_1\sigma_2\sigma_1\sigma_1 \sigma_3\sigma_1\sigma_2\sigma_3=\sigma_1\sigma_2\sigma_1\sigma_2\sigma_2\sigma_1 \sigma_3\sigma_1\sigma_2\sigma_3\\
 &=[\sigma_1, \sigma_2]\sigma_2\sigma_1\sigma_3\sigma_1\sigma_2\sigma_3=[\sigma_1, \sigma_2]\sigma_2\sigma_1\sigma_3\sigma_1 \sigma_3\sigma_3\sigma_2\sigma_3\\
 &=[\sigma_1, \sigma_2]\sigma_2[\sigma_1,\sigma_3]\sigma_3\sigma_2\sigma_3
 =[\sigma_1, \sigma_2][\sigma_1,\sigma_3][\sigma_2,\sigma_3].
\end{align*}
By the last example, $\mathcal{G}_{F_1}$, $\mathcal{G}_{F_2}$ are 
isomorphic to the group $\big(\prod_2 C_4\big)\rtimes C_2$. 
Therefore,
 \[ 
  \mathcal{G}_F\cong\Big(\big(\prod_2 C_4\rtimes C_2\big) \times\langle\mathcal{G}_{F_1},\mathcal{G}_{F_2}\rangle\Big)\rtimes \Big(\prod _2 C_4\rtimes C_2\Big).
  \]  
On the other hand, $\langle\mathcal{G}_{F_1},\mathcal{G}_{F_2}\rangle=\langle [a_i,a_j] \,|\, a_i \in \mathcal{G}_{F_1},\, a_j \in \mathcal{G}_{F_2}\rangle,$ but
 \[
\mathcal{G}_{F_1}=\langle\sigma_1,\sigma_2,\sigma_3\rangle\qquad \mathcal{G}_{F_2}=\langle\sigma_4,\sigma_5,\sigma_6\rangle.
\] 
So $\langle\mathcal{G}_{F_1},\mathcal{G}_{F_2}\rangle=\langle[\sigma_i,\sigma_j]  \,| \, i=1,2,3,\; j=4,5,6\rangle$.  
There are nine such commutators.
Since  any of $[\sigma_i,\sigma_j]$ has order two, we have: 
\[ 
\langle\mathcal{G}_{F_1},\mathcal{G}_{F_2}\rangle\cong \prod_{9}C_2, \quad
  \mathcal{G}_F\cong \mathcal{G}_{F_1}\ast \mathcal{G}_{F_2}\cong\Big(\big(\prod_2 C_4\rtimes C_2\big) \times \prod_{9}C_2\Big)\rtimes\Big(\prod _2 C_4\rtimes C_2\Big). 
\]  
Therefore $\mathcal{G}_F$  is a group of order $2^{19}$. 
~Since $\Phi_F=[\mathcal{G}_F,\mathcal{G}_F]$ there are two cases for elements of $\Phi_F$. 
~First, if $g_1,g_2$ are in different components,  there are nine such commutators. 
Second, if $g_1,g_2$ are in the same component, there are four such commutators: $[\sigma_1,\sigma_2],[\sigma_1,\sigma_3] $ of $X_1$ and  
$[\sigma_4,\sigma_5], [\sigma_4,\sigma_6]$ of $X_2$. 
~But $(\sigma_1\sigma_2\sigma_3)^2=1$ and 
$(\sigma_4\sigma_5\sigma_6)^2=1$, so $[\sigma_2,\sigma_3]$ and $[\sigma_5,\sigma_6]$ will be generated by these four commutators. 
Therefore altogether we  have 13 commutators and  any of them has order two, so $\Phi_F\cong\prod_{13}C_2$.
\end{example}

\end{document}